\documentclass[12pt]{alggeom}

\usepackage{amsmath}
\usepackage{xspace}
\usepackage{amsfonts}
\usepackage{amsthm}
\usepackage{amssymb}
\usepackage{times} 
\usepackage{graphicx}
\usepackage{hyperref}
\usepackage{enumerate}
\date{}

\numberwithin{equation}{section}
\newtheorem{thm}{Theorem}[section]
\newtheorem{prop}[thm]{Proposition}

\newtheorem{lemma}[thm]{Lemma}

\newtheorem{defn}[thm]{Definition}
\newtheorem{cor}[thm]{Corollary}
\newtheorem{remark}[thm]{Remark}

\makeatletter
\newcommand{\etale}{\'etal\@ifstar{\'e}{e\xspace}}
\makeatother

\begin{document}

\title{Azumaya algebras with involution and classical semisimple  group schemes}
\author{S. Srimathy}
\address{School of Mathematics\\
Tata Institute of Fundamental Research\\
Mumbai\\
India\\
email: srimathy@math.tifr.res.in}

\begin{abstract}
Let $S$ be a non-empty scheme with 2 invertible. In this paper we present a functor  $F: AZ_*^{n'} \rightarrow GS_*^n$  where   $AZ_*^{n'}$ and $GS_*^n$ are  fibered categories over $Sch_S$  given respectively  by   degree-$n'$ Azumaya algebras with  an involution of type $*$ and rank-$n$  adjoint group schemes  of classical type $*$ with absolutely simple fibers. Here $n'$ is a function of $n$.   We show that this functor is an equivalence of fibered categories using \etale descent, thus  giving a classification of  adjoint (as well as simply connected) group schemes over  $S$, generalizing the well known case when the base scheme is the spectrum of a field. In particular, this implies that every adjoint group scheme of classical type  with absolutely simple fibers is isomorphic to the neutral component of the  automorphism group scheme of a unique (up to isomorphism) Azumaya algebra with involution.  We also show  interesting applications of this classification such as  specialization theorem for isomorphism classes of Azumaya algebra with involution over Henselian local rings, uniqueness of integral model for groups with good reduction over discrete valued fields and discuss its implications on the Grothendieck-Serre conjecture over certain domains. 
\end{abstract}
\maketitle

\section{Introduction}
\label{sec:intro}
It is well known that the category of rank-$n$ (with some $n$ excepted) absolutely simple adjoint (or simply connected) algebraic groups of a given classical type over any field $F$ ($char~F \neq 2$) is equivalent to the category of degree-$n'$ central simple algebras with involution of analogous type over $F$. Here $n'$ is a function of $n$. Moreover, the functor which gives this equivalence is obtained by taking a given central simple algebra  with involution over $F$ to the identity component of its automorphism group (or its simply connected cover for the  simply connected case). This result is originally due to Weil (\cite{weil}) and the proof of this equivalence can also be found in \cite[\S26, Chapter VI]{boi}. This gives  neat classification results for groups of classical types in terms of central simple algebras with involution which can be translated to the well understood language of  sesquilinear forms over division algebras. This kind of classification is very useful for studying many properties of algebraic groups and the projective homogeneous varieties associated to them. \\
\indent In this paper, we show a similar  classification  for adjoint (as well as simply connected)  group schemes  over an arbitrary scheme where is $2$ invertible. Since we are in the general case of arbitrary base scheme, we use the language of stacks and gerbes to  prove  that the fibered category of degree-$n'$  Azumaya algebras with  an involution of a given type is equivalent to the  fibered category of rank-$n$ classical adjoint group schemes   with absolutely simple fibers of  the corresponding type via \etale descent. As before, $n'$ is a function of $n$ with some $n$ excepted.  This implies that an adjoint group scheme of classical type (see Definition \ref{defn:grp}) with absolutely simple fibers is isomorphic to the neutral component of  automorphism group scheme of a unique (up to isomorphism) Azumaya algebra with involution. We also give a few applications of this classification such as a specialization theorem for Azumaya algebras with involution and uniqueness of integral model for group schemes with good reduction.  Another interesting corollary is that the Grothendieck-Serre conjecture on principal $G$-bundles holds whenever $G$ is an adjoint  group scheme   with absolutely simple fibers  over $R$ where is  $R$ is a regular local ring containing a field of characteristic $\neq 2$ or $R$ is a semilocal B\'{e}zout domain with $2$ invertible.

\section{Notations}
Through out this paper, $S$ denotes a non-empty scheme with $1/2 \in \mathcal{O}_S(S)$. The category of schemes over $S$ will be denoted by $Sch_S$. Given a scheme $X$ over $S$ and a point $s$ in $S$, $k(s), X_s, X_{\overline{s}}$ denote respectively the  residue field, fiber and geometric fiber at $s$. For a presheaf $\mathcal{F}$ over a scheme $X$,  $\mathcal{F}_x$ denotes its stalk at the point $x \in X$. The ring of $n \times n$ matrices is denoted by $\mathcal{M}_n$. The identity matrix of size $n$ is denoted by $I_n$. For a sheaf of algebra $\mathcal{A}$, $\mathcal{A}^{op}$ denotes the sheaf  of opposite algebra given by $U \rightarrow \mathcal{A}(U)^{op}$. The reference  \cite{sga3} is mentioned as SGA3.

\section{Group schemes over an arbitrary scheme}
\label{sec:gpscheme}
In this section we recall the necessary results from the literature on  group schemes over an arbitrary scheme. The main sources of reference are SGA3 and \cite{demazure_red}.\\

A group scheme $G$ over $S$ is called \emph{reductive} (resp. semi-simple, adjoint, simply connected) if $G$ is affine, smooth over $S$ and for every $s \in S$, the geometric fiber $G_{\overline{s}}$ is a connected reductive (resp. semi-simple, adjoint, simply-connected) group (\cite[Def. 2.1.2]{demazure_red} and  [SGA3, Exp. XXII, Def. 4.3.3]).  \\
 
 Let $G$ be an adjoint (or simply connected) group scheme $G$ over $S$. The \emph{type} (resp. \emph{rank}) of $G$ at $s \in S$ is the type (resp. rank) of $G_{\overline{s}}$ ([SGA3, Exp. XXII, Def. 2.7]).   The type and rank of $G$ are  locally constant functions over $S$ ([SGA3, Exp. XXII, Prop.2.8] and  [SGA3, Exp. XIX, Cor. 2.6]). 
 \begin{remark}
 Any adjoint (resp. simply connected) group scheme over $S$ is isomorphic to  the Weil restriction $R_{S'/S}(G')$ of an adjoint (resp. simply connected) group scheme $G'$ with absolutely simple fibers  over $S'$ where $S' \rightarrow S$ is a finite \etale cover. Moreover, the pair $(S', G')$ is uniquely determined upto a unique $S$-isomorphism (see  \cite[Proposition 6.4.4 and Remark 6.4.5]{bconrad_red}). Therefore,  classification of adjoint group schemes over $S$ reduces to classification of adjoint group schemes with absolutely simple fibers.
 \end{remark}

\begin{defn}\label{defn:grp}
Let $G$ be an adjoint group scheme over $S$ with absolutely simple fibers. We say that $G$ is  of type $A$ (resp. $B$, $C$, $D$) and rank $n$ if every fiber  $G_{\overline{s}}, s \in S$ is of type $A$ (resp. $B$, $C$, $D$) and rank $n$. In general, we say that $G$ is of classical type if it is of type $A$, $B$, $C$ or $D$. 
\end{defn}
 
\begin{prop}
\label{prop:grp_etale}
 Let $G$ and $H$ be adjoint group schemes over $S$  with absolutely simple fibers of a given type and rank. Then locally for the \etale topology on $S$,  $G$ and $H$ are isomorphic. In fact,  any reductive group scheme with  root datum   $\mathcal{R}$ is \etale locally isomorphic to a unique Chevalley group  over $spec~ \mathbb{Z}$ with root datum $\mathcal{R}$. 
\end{prop}
\begin{proof}
See \cite[\S5, Cor. 5.1.4.I and Prop. 5.1.6]{demazure_red}.
\end{proof}

 Let $G$ be any group scheme of locally finite type over $S$. In [SGA3, Exp. VI\textsubscript{B}, Def 3.1], the notion of \emph{the neutral component of $G$} denoted by $G^0$ is defined by the  functor 
\begin{align*}
    T \rightarrow G^0(T) = \{u \in G(T)| \forall s \in S ,u_s(T_s) \subset (G_s)^0\}
\end{align*}
We refer the reader to [SGA3, \S3, Exp. VI\textsubscript{B}] for more details. It can be  shown that  the neutral component of $G$ is stable under base change. We restate here below.
\begin{prop} ([SGA3, Exp. VI\textsubscript{B}, Prop. 3.3]).
\label{prop:basechange}
Let $G$ be group scheme over $S$. Then for any scheme $S' \rightarrow S$, we have 
\begin{align}
    (G \times_S S')^0 = G^0 \times_S S'
\end{align}
i.e., the functor $G \rightarrow G^0$ commutes with base change.
\end{prop}

We will be using the following result about $G^0$.
\begin{prop} (\cite[Prop 3.1.3]{bconrad_red})
\label{prop:conrad_neutral}
Let $G$ be a smooth separated group scheme of finite presentation such that $G_{\overline{s}}^0$ is reductive for all $s \in S$. Then $G^0$ is a reductive  group scheme over $S$ that is open and closed in $G$.
\end{prop}

For a group scheme $G$ over $S$, we  associate the  automorphism functor  $\underline{Aut}(G)$  on $Sch_S$ defined by
\begin{align*}
    \underline{Aut}(G): S' \mapsto Aut_{S'-grp}(G_{S'})
\end{align*}

\begin{prop}
\label{prop:autg}
Let $G$ a semisimple group scheme over $S$. Then the functor $\underline{Aut}(G)$ is represented by a smooth, affine scheme over $S$.
\end{prop}
\begin{proof}
See [SGA 3, Exp. 24, Theorem 1.3(i) and Corollary 1.6]. 
\end{proof}

\begin{remark}
\label{rmk:representable}
A representable functor  on $Sch_S$ is a sheaf for the fpqc topology (\cite[Theorem 2.55]{vistoli_descent}). Therefore if $G$ is semisimple, by Proposition \ref{prop:autg}, $\underline{Aut}(G)$ is a sheaf for the fpqc topology and hence for the \etale topology on $Sch_S$.
\end{remark}

\section{Azumaya algebras with involution over an arbitrary scheme}
 Recall that an \emph{Azumaya algebra} $\mathcal{A}$ over $S$ is an $\mathcal{O}_S$-algebra that is locally free and finite type as an $\mathcal{O}_S$-module such that the canonical homomorphism 
 \begin{align*}
     \mathcal{A} \otimes \mathcal{A}^{op} &\rightarrow End_{\mathcal{O}_S-mod}(\mathcal{A}) \\
     a \otimes b &\mapsto (x \mapsto a \cdot x \cdot b)
 \end{align*}
is an isomorphism. This implies that there is an \etale covering $\{U_i \rightarrow S \}$ such that $\mathcal{A} \otimes_{\mathcal{O}_S}\mathcal{O}_{U_i}  \simeq \mathcal{M}_{n_i}(\mathcal{O}_{U_i})$ for some $n_i$. If   $n_i = n$ for all $i$ (this happens for example when $S$ is connected), we call $\mathcal{A}$ an Azumaya algebra  of \emph{degree} $n$ over $S$. Also recall that for an Azumaya algebra $\mathcal{A}$ over $S$, $\mathcal{A}_s \otimes k(s)$ is a central simple algebra over $k(s)$ for every $s \in S$ (see \cite[Chapter IV, \S2, Prop. 2.1]{milne_etale}).\\

Involutions on central simple algebras are well studied in the literature (\cite{boi}). In a similar fashion one can also define involution on  Azumaya algebras over the scheme $S$.  This is discussed in detail in \cite{parimala_involution} and \cite{parimala_brauer} which we briefly recall now. A classification of involutions on Azumaya algebras  in a more general setting  can be found  in  \cite[\S 5]{azumaya_classification}.

An \emph{involution of first kind} $\sigma$ on $\mathcal{A}$ is an isomorphism of $\mathcal{O}_S$-algebras
\begin{align*}
    \sigma: \mathcal{A} \rightarrow \mathcal{A}^{op}
\end{align*}
 such that $\sigma^{op}\circ \sigma$ is the identity. The involution $\sigma$ on $A$ is said to be of \emph{orthogonal type} (resp. \emph{symplectic type}) if \etale locally on $S$, it corresponds to a non-degenerate  symmetric (resp. skew-symmetric) bilinear form on a locally free $\mathcal{O}_S$-module with values in a line bundle over $S$ (see   \cite[\S1.1]{parimala_brauer}).\\
 
 Now let $\pi: T \rightarrow S$ be a an \etale covering of degree-$2$. Let $\mathcal{A}$ be an Azumaya algebra over $T$. An involution $\sigma$ of \emph{second kind} (a.k.a \emph{unitary type}) on $\mathcal{A}$ is an anti-automorphism  on $\mathcal{A}$ of order $2$ which on $\mathcal{O}_T$, restricts to the non-trivial element of the  Galois group of the covering.  Locally for the \etale topology on $S$ unitary involutions correspond to hermitian forms on locally free $\mathcal{O}_T$-modules (\cite[\S1.2]{parimala_brauer}).\\
 
 \noindent  The data of an Azumaya algebra $\mathcal{A}$ with a involution $\sigma$ over $S$ is denoted by $(\mathcal{A}, \sigma)$. A homomorphism between Azumaya algebras with involution is a homomorphism between the algebras which respects the involution structure.  \\
 
\begin{defn}
 In the case above where $\mathcal{A}$ is an Azumaya algebra over a quadratic \etale extension $T$ of $S$ and $\sigma$ is unitary, we will make a slight abuse of notation and call $(\mathcal{A},\sigma)$ an \emph{Azumaya algebra with unitary involution over $S$}  even though the center of $\mathcal{A}$ is not $\mathcal{O}_S$. This agrees with the corresponding notion of central simple algebras with unitary involution defined in  \cite[Chapter I, \S2.B]{boi}.  
\end{defn}

\begin{remark}
 If $(A, \sigma)$ is a degree-$n$ Azumaya algebra with involution over $S$ where $\sigma$ is of a given type, then for any $s \in S$, $(A_s \otimes k(s), \sigma_s \otimes k(s))$ is a  degree-$n$ central simple algebra with involution of the same type over $k(s)$ 
\end{remark}

 We now give an \etale local description of Azumaya algebras with involution.
 
 \begin{prop}
 \label{prop:local_azumaya}
 Let $(A, \sigma)$ be a degree-$n$ Azumaya algebra with involution over  $S$. Then locally for the \etale topology on $S$,  we have
 \begin{enumerate}[{(1)}]
     \item  $(\mathcal{A}, \sigma) \simeq (\mathcal{M}_n, tr)$ when $\sigma$ is of orthogonal type  where $tr: A \rightarrow A^{tr}$ is the transpose involution.
     \item $(\mathcal{A}, \sigma) \simeq (\mathcal{M}_{2m}, sp)$ when $\sigma$ is of symplectic type. Here  $sp$ is the involution on $\mathcal{M}_{2m}$ given by $A \rightarrow J_m A^{tr} J_m^{-1}$ where $J_m = \begin{bmatrix} 0& I_m \\ -I_m & 0 \end{bmatrix}$ is the standard matrix associated to an alternating form.
     \item $(\mathcal{A}, \sigma) \simeq (\mathcal{M}_n \times \mathcal{M}_n^{op}, \epsilon)$ when $\sigma$ is of unitary type where $\epsilon : (A,B^{op}) \rightarrow (B,A^{op})$ is the exchange involution.
 \end{enumerate}
 \end{prop}
 \begin{proof}
 All of the above are well known if the base scheme $S$ is the spectrum of a field (see \cite{boi}). For the general case, proofs of (1) and (2) can be easily derived and can  also be found in the literature. See for example, \cite[\S1.1]{parimala_brauer}  and \cite[Chapter III, \S8.5]{knus_quad}. We could not find the proof of (3) anywhere in the literature, so we give a proof here. \\
 \indent In this case,  $\mathcal{A}$ an Azumaya algebra with center  $\mathcal{O}_T$ where $T \rightarrow S$  is a degree-$2$ \etale cover of $S$ with the non-trivial element in its Galois group denoted by $\tau$. Consider the  degree-$2$ \etale cover $p:T \times_S T \rightarrow T$ obtained via the \etale base change $ \pi: T \rightarrow S$. It follows from Galois theory of schemes that $(T \times_S T, \pi^*\tau)  \simeq (T \coprod T, ex)$ (see the proof of Theorem 5.10 in \cite{galois_schemes}) where $ex: (x,y) \rightarrow (y,x)$. Now $p^*\mathcal{A}$  is an Azumaya algebra over $T \coprod T$.  Therefore the center $\mathcal{O}_T \times \mathcal{O}_T$ of $ p^*\mathcal{A}$ contains the idempotent $i= (1,0)$ where $ex(i) = 1-i$. Note that $B =i(p^*\mathcal{A})$  an Azumaya algebra over $T$  and we have an isomorphism of Azumaya algebras with unitary involution given by 
 \begin{align*}
     (p^*\mathcal{A}, p^*\sigma) &\xrightarrow{\simeq} (B \times B^{op}, \epsilon) \\
      a &\mapsto (i\cdot a, (i\cdot(p^*{\sigma}(a)))^{op})
 \end{align*}
 By taking a suitable \etale covering of $T$ which splits $B$, we obtain (3).
 \end{proof}
 
\begin{defn}
With notations as in Proposition \ref{prop:local_azumaya},  we say that  a degree-$n$ Azumaya algebra with unitary involution (resp. orthogonal, resp. symplectic)  over  $S$ is   split if it is isomorphic to  $(\mathcal{M}_n \times \mathcal{M}_n^{op}, \epsilon)$ (resp. $(\mathcal{M}_n, tr)$, resp. $(\mathcal{M}_n, sp)$).  
\end{defn}
 
 \section{The Group scheme of Automorphisms of Azumaya algebras with involution}
  Let $(\mathcal{A}, \sigma)$ be a degree-$n$ Azumaya algebra with involution  of any type  over $S$. Consider the functor 
  \begin{align*}
      \underline{Aut}(\mathcal{A}, \sigma) : &Sch_S \rightarrow Groups \\
      &(U \xrightarrow{i} S) \mapsto Aut_{\mathcal{O}_U-alg}(i^*\mathcal{A}, i^*\sigma)
  \end{align*}
 where $Aut_{\mathcal{O}_U-alg}(i^*\mathcal{A}, i^*\sigma)$ is the group of $\mathcal{O}_U$-algebra automorphisms of $i^*\mathcal{A}$ compatible with the involution $i^*\sigma$.\\
 
 \noindent \textbf{Notation:}   If a functor $\underline{F}$  on $Sch_S$ is representable, let us  denote the representing scheme by $F$.

\begin{thm}
\label{thm:autoazu}
 The functor $\underline{Aut}(\mathcal{A}, \sigma)$ is representable by a  smooth, affine group scheme over $S$.
\end{thm}
 \begin{proof}
 Let $S' \rightarrow S$ be an fpqc morphism. Recall that the category of affine $S$-schemes  is equivalent to the category of affine $S'$-schemes with descent data. To see this use  \cite[Theorem 4, Chapter 6]{neron} and  the fact that for any scheme $X$, the category of affine $X$-schemes is anti-equivalent to the category of quasi-coherent sheaves of $\mathcal{O}_X$-algebras (\cite[\href{https://stacks.math.columbia.edu/tag/01S5}{Tag 01S5}, Lemma 29.11.5]{stacks-project} or  see   \cite[Theorem 4.33]{vistoli_descent}).  Moreover, the properties smooth and affine are fpqc local over the base (see \cite[\href{https://stacks.math.columbia.edu/tag/02YJ}{Tag 02YJ}]{stacks-project} or \cite[Prop 2.7.1 and Prop 6.8.3]{ega4_2}).   Therefore  by Proposition \ref{prop:local_azumaya} it suffices to prove the theorem when $(\mathcal{A}, \sigma)$ is split. So assume that $(\mathcal{A}, \sigma)$ is split. Note that in this case the functor
 \begin{align*}
     \underline{Aut}(\mathcal{A}) : (U \xrightarrow{i} S) \mapsto Aut_{\mathcal{O}_U-alg}(i^*\mathcal{A})
 \end{align*}
 is representable by a closed  subscheme of the affine $\mathbb{Z}$-scheme $End_{\mathcal{O}_S -mod}(\mathcal{A}) \simeq \mathcal{M}_{r}$  where $r = dim_{\mathcal{O}_S}(\mathcal{A})$ (see  \cite[\S2, Chapter IV]{milne_etale}). Module homomorphisms of $\mathcal{A}$ that respect the involution can be expressed as vanishing of  polynomials and hence is representable by a closed subscheme of  $\mathcal{M}_{r^2}$. The intersection of these two subschemes represents  $\underline{Aut}(\mathcal{A}, \sigma)$  and hence is an affine scheme over $S$.  \\
 Now we prove  that this affine scheme  is smooth. For the unitary case, consider the functor 
 \begin{align*}
     \underline{Aut}_{\mathcal{O}_S \times \mathcal{O}_S} (\mathcal{M}_n \times \mathcal{M}_n^{op}, \epsilon) : (U \xrightarrow{i} S) &\mapsto Aut_{\mathcal{O}_U \times \mathcal{O}_U -alg}(\mathcal{M}_n(\mathcal{O}_U) \times \mathcal{M}_n(\mathcal{O}_U)^{op}, \epsilon) \\
     &\simeq Aut_{\mathcal{O}_U -alg}(\mathcal{M}_n(\mathcal{O}_U))
 \end{align*}
 Hence $\underline{Aut}_{\mathcal{O}_S \times \mathcal{O}_S} (\mathcal{M}_n \times \mathcal{M}_n^{op}, \epsilon)$ is representable by a smooth affine $\mathbb{Z}$-scheme  (see  \cite[Chapter IV, \S2]{milne_etale} for affineness and \cite[Chapter II, \S5, 2.7]{demazure_gabriel} for smoothness). 
The functor $\underline{Aut}_{\mathcal{O}_S} (\mathcal{O}_S \times \mathcal{O}_S)$ is representable by the finite group scheme $\mathbb{Z}/2\mathbb{Z}$.  Since  representable functors on $Sch_s$ are sheaves for the fpqc topology (\cite[Theorem 2.55]{vistoli_descent}), we have the following exact sequence of  sheaves over $S$
 \begin{align}
     1 \rightarrow \underline{Aut}_{\mathcal{O}_S \times \mathcal{O}_S} (\mathcal{M}_n \times \mathcal{M}_n^{op}, \epsilon) \rightarrow \underline{Aut}(\mathcal{M}_n \times \mathcal{M}_n^{op}, \epsilon) \rightarrow \underline{Aut}_{\mathcal{O}_S} (\mathcal{O}_S \times \mathcal{O}_S) \rightarrow 1
 \end{align}
 Smoothness of $Aut(\mathcal{M}_n \times \mathcal{M}_n^{op}, \epsilon)$ now follows from [SGA3, Exp. VI\textsubscript{B}, Prop. 9.2].\\
\indent For the other cases, note that $Aut(\mathcal{M}_{n}, sp) \simeq Sp_n/Z$ and $Aut(\mathcal{M}_n, tr) \simeq O_n/Z$ (as fpqc quotients) where $Sp_n$ , $O_n$ denote  respectively the  symplectic group, orthogonal group and $Z$ their respective centers. Required smoothness results for these group schemes  follow from  \cite[Chapter II, \S5, 2.7]{demazure_gabriel} and  [SGA3, Exp. VI\textsubscript{B}, Prop. 9.2].
 \end{proof}
 
\begin{remark}
By   \cite[Theorem 2.55]{vistoli_descent} and Theorem \ref{thm:autoazu},  $\underline{Aut}(\mathcal{A}, \sigma)$ is a sheaf   for the fpqc topology.
\end{remark}
 
\begin{defn} Let $Aut^0(\mathcal{A}, \sigma) := Aut(\mathcal{A}, \sigma)^0$ denote the neutral component of $Aut(\mathcal{A}, \sigma)$  defined in \S\ref{sec:gpscheme}. 
 \end{defn}
 
 \begin{thm}
 \label{thm:neutral}
 Let $(\mathcal{A}, \sigma)$  be a degree $n$  Azumaya algebra with involution  over $S$ where $n\neq 2,4$ whenever $\sigma$ is orthogonal.   Then   $Aut^0(\mathcal{A}, \sigma)$ is an adjoint group scheme  over $S$  with absolutely simple fibers. Moreover  $Aut^0(\mathcal{A}, \sigma)$ is of type $A$ if $\sigma$ is unitary, type $B$ if $\sigma$ is orthogonal and $n$ is odd, type $C$ if $\sigma$ is symplectic and type $D$ if $\sigma$ is orthogonal and $n$ is  even.
 \end{thm}
 \begin{proof}
By Lemma \ref{lem:pullback} below, for every $s \in S$ we have 
 \begin{align}
 \label{eqn:neutral}
     Aut^0(\mathcal{A}, \sigma)_s = (Aut(\mathcal{A}, \sigma)_s)^0 \simeq  Aut(\mathcal{A}_{k(s)}, \sigma_{k(s)})^0 = Aut^0(\mathcal{A}_{k(s)}, \sigma_{k(s)})
 \end{align}
Now $Aut^0(\mathcal{A}_{k(s)}, \sigma_{k(s)})$ is an absolutely simple adjoint algebraic group over  $k(s)$ (\cite[Chapter  VI, \S26]{boi}). This together with Theorem \ref{thm:autoazu} and Proposition \ref{prop:conrad_neutral} shows that  $Aut^0(\mathcal{A}, \sigma)$ is a  reductive group scheme over $S$. The rest of the claim follows from (\ref{eqn:neutral}), unravelling the definitions in \S\ref{sec:gpscheme} and   \cite[Chapter  VI, \S26]{boi}.
 \end{proof}

 \begin{remark}
   In the above theorem,  when $n=2$, $Aut^0(\mathcal{A}, \sigma)$ is not adjoint and when $n=4$, the fibers of $Aut^0(\mathcal{A}, \sigma)$ are not absolutely simple.
 \end{remark}
 
\begin{defn}
Based on Theorem \ref{thm:neutral}, we say that the pair $(\mathcal{A}, \sigma)$ is of type $*$ where $*$ is 
 \begin{itemize}
     \item  A if $\sigma$ is unitary
     \item B if  $\sigma$ is orthogonal and the degree of $\mathcal{A}$ is odd
     \item C if $\sigma$ is symplectic
     \item D if $\sigma$ is orthogonal and the degree of $\mathcal{A}$ is even
 \end{itemize}
\end{defn}

 \section{The fibered categories $Az_*^n$ and $GS_*^n$}
 We will now construct the fibered categories of degree-$n$ Azumaya algebras with involution of a given type  and the fibered category  of rank-$n$ adjoint group schemes with absolutely simple fibers
 of a given type over $Sch_S$ and show that they are stacks (in fact gerbes) with respect to the \etale topology. The classical reference for stacks and gerbes is \cite{giraud}. Other references are   \cite{vistoli_descent}, \cite{moerdijk} and the Appendix in \cite{deligne_tannakian}. 
 
\begin{defn}
Given a scheme $U$ in the category $Sch_S$, let $Az_*^n(U)$ denote the groupoid of degree-$n$ Azumaya algebras with involution of type $*$ (where $*$ is A, B, C or D)  over $U$ (morphisms are isomorphisms  of Azumaya algebras over $U$ that respect the involution structures).  For any morphism $f: V \rightarrow U$ in $Sch_S$ we have a pull-back functor 
 \begin{align*}
     Az_*^n (U) &\rightarrow Az_*^n(V)\\
        (\mathcal{A}, \sigma) &\mapsto f^*(\mathcal{A}, \sigma) := (f^*\mathcal{A}, f^*\sigma)
 \end{align*}
 which  makes the assignment  $U \rightarrow Az_*^n(U)$ a pseudo-functor  on $Sch_S$ (see \cite{vistoli_descent}, Chapter 3). 
 \end{defn}
 
\begin{defn} The  fibered category of degree-$n$ Azumaya algebras with involution of type $*$ associated to the above pseudo-functor is denoted by $Az_*^n \rightarrow Sch_S$.
\end{defn}
 
\begin{defn} The fibered category $GS_*^n \rightarrow Sch_S$ of rank-$n$ adjoint group schemes with absolutely simple fibers of type $*$ (where $*$ is A, B, C or D) is defined in a similar way where morphisms in  every fiber $GS_*^n(U)$  are isomorphisms.
 \end{defn}
 
 \begin{prop}
 \label{prop:gerbe}
 The fibered categories $Az_*^n \rightarrow Sch_S$  and  $GS_*^n \rightarrow Sch_S$ are stacks for the \etale topology. In fact, they are gerbes.
 \end{prop}
 \begin{proof}
 By standard arguments from decent theory, it is easy to see that  $GS_*^n \rightarrow Sch_S$ is a stack for the \etale topology. It is a gerbe by Proposition \ref{prop:grp_etale}. For the case of $Az_*^n$,  we note from descent theory that quasi-coherent sheaves (as well as the morphisms) satisfy descent for the fpqc topology (\cite[Chapter 6, Theorem 4]{neron}) and hence also for the \etale topology. Quasi-coherent sheaves together with additional structure such as an algebra structure and involutions also descend (see for example   \cite[\S4.2.2 and \S4.2.3]{vistoli_descent} or   \cite[Chapter II, Theorem 3.4]{azumaya_descent}). This shows that $Az_*^n\rightarrow Sch_S$ is a stack for the \etale topology. It is a gerbe by Proposition \ref{prop:local_azumaya}. 
 \end{proof}

\section{Equivalence of $Az_*^{n'}$ and $GS_*^n$}
Let $Az_*^n \rightarrow Sch_S$  and  $GS_*^n \rightarrow Sch_S$ be the fibered categories defined in the previous section. In this section we describe a morphism between $Az_*^{n'}$ and $GS_*^n$ (where $n'$ is determined by $n$) that will yield the required equivalence of  fibered categories.  We will need the following lemma.

\begin{lemma}
\label{lem:pullback}
Let $(\mathcal{A}, \sigma) $ be an Azumaya algebra  with involution  over a scheme $X$. The assignment 
\begin{align*}
(\mathcal{A}, \sigma) \rightarrow Aut^0(\mathcal{A}, \sigma)
\end{align*}
respects pull-backs i.e., for $f: Y \rightarrow X$ we have a canonical isomorphism
\begin{align*}
    f^*(Aut^0(\mathcal{A}, \sigma)) \simeq Aut^0(f^*(\mathcal{A}, \sigma))
\end{align*}
\end{lemma}
\begin{proof}
We note that  $f^*(\underline{Aut}(\mathcal{A}, \sigma)) = \underline{Aut}(f^*(\mathcal{A}, \sigma))$. Hence by the Yoneda lemma, there is a canonical isomorphism $f^*(Aut(\mathcal{A}, \sigma)) \simeq Aut(f^*(\mathcal{A}, \sigma))$. This together with Proposition \ref{prop:basechange} proves the claim.
\end{proof}

\begin{thm}
\label{thm:main}
The functor 
\begin{align*}
\mathbf{Aut^0} : Az_*^{n'} &\rightarrow GS_*^n \\
         (\mathcal{A}, \sigma) &\mapsto Aut^0(\mathcal{A}, \sigma)\\
          ((\mathcal{A}, \sigma) \xrightarrow{i}  (\mathcal{B}, \tau)) &\mapsto (\phi \rightarrow i\circ \phi \circ i^{-1})
\end{align*}
is an equivalence of fibered categories where 
\begin{itemize}
    \item $n>1$ and $n' = n+1$ if $*$ is of type $A$ 
    \item $n'=2n+1$ if $*$ is of type $B$
    \item  $n' = 2n$ if $*$ is of type $C$
    \item $n > 2$, $n\neq 4$ and $n'= 2n$ if $*$ is of type $D$
\end{itemize}
\end{thm}
\begin{proof}
The functor $\mathbf{Aut^0}$ defines a morphism of fibered categories by Lemma \ref{lem:pullback}.  Since $Az_*^{n'} \rightarrow Sch_S$  and  $GS_*^n \rightarrow Sch_S$ are gerbes by Proposition \ref{prop:gerbe}, to show that $\mathbf{Aut^0}$ is an equivalence it suffices to show that for any  object $(\mathcal{A}, \sigma)$  in $Az_*^{n'}(S)$ (say the split object), $\mathbf{Aut^0}$  induces isomorphism of  sheaves $\underline{Aut}(\mathcal{A}, \sigma)$ and $\underline{Aut}(Aut^0(\mathcal{A}, \sigma))$ (see \cite[Chapter IV, \S3.1]{giraud} or \cite{moerdijk}). Now  by Theorem \ref{thm:autoazu} and Proposition \ref{prop:autg} both $\underline{Aut}(\mathcal{A}, \sigma)$ and $\underline{Aut}(Aut^0(\mathcal{A}, \sigma))$ are represented  by smooth affine group schemes over $S$ denoted  by $Aut(\mathcal{A}, \sigma)$ and $Aut(Aut^0(\mathcal{A}, \sigma))$ respectively. So it suffices to check that $\mathbf{Aut^0}$ induces isomorphism at every  fiber (\cite[\href{https://stacks.math.columbia.edu/tag/039E}{Tag 039E}]{stacks-project}, \cite[\href{https://stacks.math.columbia.edu/tag/025G}{Tag 025G}]{stacks-project}). Again by the proof of Lemma \ref{lem:pullback}, we see that for every  $s \in S$, $\mathbf{Aut^0}$ induces morphism of schemes over $k(s)$
\begin{align*}
\mathbf{Aut^0}_{{k(s)}} : {Aut}(\mathcal{A}_{{k(s)}}, \sigma_{{k(s)}}) \rightarrow  {Aut}(Aut^0(\mathcal{A}_{{k(s)}}, \sigma_{{k(s)}}))
\end{align*}
Now we are in the case of fields and the fact that $\mathbf{Aut^0}_{{k(s)}}$ is an isomorphism follows from \cite[Chapter  VI, \S26]{boi}.
\end{proof}

\begin{remark}
    In the case of $n=1$ for type $A$, the above equivalence holds if $Az^2_A$ is defined to be the fibered category of Azumaya algebras of degree $2$ (\cite[page 366]{boi}). Similarly, in the case of  $n=2$ for type $D$, since the corresponding root system is not irreducible, one has to  remove the phrase ``with absolutely simple fibers" in the definition of  $GS^D_2$ for the above equivalence to hold. The proofs are similar. 
\end{remark}

\begin{remark}
The classification of simply connected group schemes  with absolutely simple fibers over $S$ is similar to the adjoint case. The functor in this case is given by 
\begin{align*}
    \mathbf{\overline{Aut^0}}: (\mathcal{A}, \sigma) \rightarrow \overline{Aut^0 (\mathcal{A}, \sigma)}
\end{align*}
where for a semisimple group scheme $G$ over $S$, $\overline{G}$ denotes its unique simply connected cover (see  \cite[Exercise 6.5.2]{bconrad_red}). 
\end{remark}

\begin{remark}
  Recall that when the base scheme is the spectrum of a field of characteristic $2$, the groups  $Aut^0(A, \sigma)$ are not smooth  when $\sigma$ is orthogonal (\cite[page 347]{boi}). Hence when $2$ is  not invertible in $S$, things look a little different for types $B$ and $D$. One could possibly work fppf locally instead of etale locally  for these types to get a similar classification as in the case of fields (\cite[Lemma C.2.1]{conrad_reductive}).  We did not want to deal with these subtle technicalities in this paper that will take away the nice picture when  $2$ is invertible. 
\end{remark}

\section{Applications}
In this section we show some interesting consequences of Theorem \ref{thm:main}. \\

For a ring $R$, let $P_{n'}(R,*)$ denote the isomorphism classes of  degree-$n'$  Azumaya algebras with involution  of type $*$ over $R$.
\begin{cor}
Let $R$ be a Henselian local ring with residue field $k$. Assume $char~k \neq 2$.  Then the  restriction map
\begin{align*}
    P_{n'}(R,*) \rightarrow P_{n'}(k, *)
\end{align*}
is bijective for every $(n',*)$ listed in Theorem \ref{thm:main}. 
\end{cor}
\begin{proof}
This follows from Lemma \ref{lem:pullback}, Theorem \ref{thm:main} and [SGA3, Exp. XXIV, Prop. 1.21].
\end{proof}
\begin{remark}
The above corollary is a generalization of a similar statement about  isomorphism classes of Azumaya algebras. See Theorem 6.1 in \cite{grothendieck_brauer}.
\end{remark}

We will recall the concept of \emph{good reduction} for algebraic groups defined in \cite{crr_spinor}. Let $k$ be a discrete valued field with valuation $v$. Let $k_v$,  $\mathcal{O}_v$ and $k^{(v)}$ denote respectively  the completion of $k$, the valuation ring of $k_v$  and the residue field. An absolutely almost simple  linear algebraic group $G$ over $k$ is said to have \emph{good reduction at} $v$ if there exists a reductive group scheme $\mathcal{G}$ over $\mathcal{O}_v$ such that  $\mathcal{G} \otimes_{\mathcal{O}_v} k_v \simeq G \otimes_{k} k_v$. In this case, let us call $\mathcal{G}$ an $\mathcal{O}_v$-\emph{model for} $G$. Studying good reduction of $G$ has many applications such as computing the genus of algebraic group and proving Hasse principles. We refer the reader to \cite{crr_spinor} for more details. 
\begin{cor}
Let $G$ be a rank-$n$  absolutely simple adjoint (or simply connected) algebraic group of classical type ($n\neq 4$ if $G$ is of type $D$) over a discrete valued field $k$ with valuation $v$. Assume  that the characteristic of the residue field $k^{(v)}$  is different from $2$.  Suppose $G$ has good reduction at $v$. Then any two $\mathcal{O}_v$-models of $G$ are isomorphic. In other words,  an  $\mathcal{O}_v$-model of $G$ if it exists is unique up to isomorphism.
\end{cor}
\begin{proof}
This follows from Lemma \ref{lem:pullback}, Theorem \ref{thm:main} and   \cite[Theorem 3.7]{sofie}.
\end{proof}

\begin{remark}
The classification in Theorem \ref{thm:main} can be used to translate  the notion of good reduction of absolutely simple adjoint (or simply connected) classical algebraic groups to the notion of good reduction of the underlying sesquilinear forms over division algebras. See for example \cite[Remark 4.3]{srimathy_hermitian}.
\end{remark}

Let $R$ be a regular local ring and let $G$ be a reductive groups scheme over $R$. A conjecture of Grothendieck and Serre (\cite[pp. 26-27, Remark 3]{gro58}, \cite[Remark 1.11(a)]{gro68} and \cite[pp. 31, Remark]{serre58}) states that rationally trivial principal $G$-homogeneous spaces are trivial i.e., the kernel of the canonical map
\begin{align} \label{eqn:conj}
    H^1(R, G) \rightarrow H^1(K, G)
\end{align}
where $K$ is the fraction field of $R$, is trivial. While the conjecture is still open to be proved in complete generality,  the proofs for various cases of $R$ and $G$ have been established since Nisnevich's thesis (\cite{nis}). If $R$ is a regular local ring containing a field of characteristic $\neq 2$ or if $R$ is a semilocal B\'{e}zout domain with $2$ invertible, I.Panin (\cite[Theorem 1.1]{panin_purity}) and S.Beke (\cite[Theorem 3.7]{sofie}) respectively have proved that any two Azumaya algebras with involutions over $R$ that are rationally isomorphic (i.e., isomorphic over the fraction field of $R$) are already isomorphic. This implies that  for the above cases of $R$, the kernel of the map in (\ref{eqn:conj}) is trivial when $G \simeq Aut(A, \sigma)$, the automorphism scheme of an Azumaya algebra with involution over $R$. Since any rank-$n$ adjoint group scheme of classical type ($n \neq 4$ if the group is of type $D$) over $R$ with absolutely simple fibers is isomorphic to  $Aut^0(A, \sigma)$  for some $(A, \sigma)$ by Theorem \ref{thm:main}, we conclude:
\begin{cor}
Let $R$ be a integral domain with $2$ invertible  that satisfies the property that any  Azumaya algebra with involution over $R$ that is split over the fraction field of $R$ is  already split. Then the Grothendieck-Serre conjecture is true for any rank-$n$ adjoint group scheme of classical type ($n \neq 4$ if the group is of type $D$)   with absolutely simple fibers over $R$. In particular, this happens when  $R$ is a regular local ring containing a field of characteristic $\neq 2$ or when $R$ is  a semilocal B\'{e}zout domain with $2$ invertible. 
\end{cor}

\section{Acknowledgements}
The author acknowledges the support of the DAE, Government of India, under Project Identification No. RTI4001. She thanks Jonathan Wise for pointing to an important reference and also thanks the anonymous referee for their useful comments and suggestions.

\nocite*{}
\bibliographystyle{alpha}
\bibliography{ref_azumaya}
\end{document}